\documentclass{amsart}

\title[ON QUASICONFORMAL EQUIVALENCE BETWEEN ...]{ON QUASICONFORMAL EQUIVALENCE BETWEEN CERTAIN INFINITELY OFTEN PUNCTURED PLANES}
\author[H. FUJINO]{ {\tiny By} \vspace{1ex} \\  HIROKI\ \ \  FUJINO \vspace{-0ex}}
\address{Graduate~School~of~Mathematics, Nagoya~University, Furo-cho Chikusa-ku Nagoya 464-8602, Japan}
\email{m12040w@math.nagoya-u.ac.jp}
\subjclass[2010]{Primary~51M04, Secondary~51M05.}
\keywords{quasiconformal mappings, uniform domain, extremal distance.}
\date{\today}

\usepackage[dvips]{graphicx}
\usepackage{amsmath}
\usepackage{amssymb}
\usepackage{amsrefs}
\usepackage{ascmac}
\usepackage{mathrsfs} 
\usepackage{setspace}

\makeatletter
\renewcommand{\subsection}{%
  \@startsection{subsection}%
   {1}%
   {4ex}%
   {4ex \@plus -1ex \@minus -.2ex}%
   {-2ex \@plus.2ex}
   {\normalfont\normalsize\bfseries}%
}%
\makeatother

\newtheorem{thm}{Theorem}[section]

\newtheorem{lemma}[thm]{Lemma}

\newtheorem{coro}[thm]{Corollary}

\newtheorem{theo}{定理}[section]  
\newtheorem{thma}[theo]{Theorem}  

\newtheorem{dfnemp}{Definition}

\newtheorem{probemp}{Problem}    
  
\newtheorem{coroemp}{Corollary}   

\newtheorem{thea}{定理}[section]  
\newtheorem{conjecturer}[thea]{Conjecture}  

\newcommand{\tei}{Teichm\"uller\ } 

\newcommand{\edc}[2]{\delta^{\hat{\mathbb{C}}}(#1,#2)}
\newcommand{\edd}[2]{\delta^{D}(#1,#2)}

\newcommand{\zokud}[2]{\mathscr{F}^{D}(#1,#2)}

\newcommand{\mode}{{\rm mod}}

\begin{document}

\begin{abstract}

A closed discrete subset $A\subset \mathbb{C}$ is called ``tame'' if $\mathbb{C}\setminus A$
is quasiconformally equivalent to $\mathbb{C}\setminus \mathbb{Z}$. By giving 
several criteria for $A$ to be tame, we shall show that $\mathbb{Z}+i\mathbb{Z}$
is not tame.

\end{abstract}

\maketitle　
\vspace{-3ex}

\section{Introduction}
\label{chap1}
\setcounter{page}{1}
\renewcommand{\thepage}{\arabic{page}}

Let $R$ be a Riemann surface.
The \tei space $T(R)$ is a space which describes all quasiconformal deformations of $R$.
It is well known that $T(R)$ becomes either a finite dimensional complex manifold or a
non-separable infinite dimensional Banach analytic manifold.
$T(R)$ becomes finite dimensional if and only if $R$ is of finite type.
Through the investigation of quasiconformal deformations of a certain infinite type Riemann surface, 
a certain characteristic subspace will be found, which is separable.\vspace{1ex}

The universal \tei space $T(\mathbb{D})$ simultaneously describes all quasiconformal deformations
of all hyperbolic type Riemann surfaces. This arises from the fact that 
each covering $X\rightarrow Y$ induces an embedding of $T(Y)$ into $T(X)$.
On the other hand, $\mathbb{C}\setminus\mathbb{Z}$ covers a certain $n$-punctured 
Riemann sphere for each $n\geq 3$.
Namely $T(\mathbb{C}\setminus\mathbb{Z})$ simultaneously describes all quasiconformal 
deformations of Riemann surfaces of genus $0$ with at least three punctures.
Needless to say, the universal \tei space $T(\mathbb{D})$ also describes them.
However for the reasons mentioned below, 
the \tei space $T(\mathbb{C}\setminus\mathbb{Z})$ is more suitable to describe them than $T(\mathbb{D})$.

For each positive integer $n$, let $R_n=(\mathbb{C}\setminus \mathbb{Z})/\langle z+n\rangle$. 
$R_n$ is an $(n+2)$-punctured Riemann
sphere, and the projection $p^n:\mathbb{C}\setminus\mathbb{Z} \rightarrow R_n$ induces the
embedding $p^n_{\ast}:T(R_n)\hookrightarrow T(\mathbb{C}\setminus \mathbb{Z})$.
The covering transformation group of $p^n$ is the cyclic group $\langle z+n\rangle$, so that,
quasiconformal deformations of $R_n$ correspond to periodic quasiconformal
deformations of $\mathbb{C}\setminus\mathbb{Z}$ with only a period $z+n$.
Then it is shown from McMullen's theorem in \cite{mcm1} that $p^n_{\ast}$
is totally geodesic for \tei metric. By contrast, the embedding of $T(R_n)$ into $T(\mathbb{D})$ is not 
totally geodesic (cf. The Kra--McMullen theorem \cite{mcm1}).
Additionally, 
$\mathbb{C}\setminus \mathbb{Z}$ is considered to be one of the smallest Riemann surface which has the above properties, that is,
there exists no Riemann surface except $\mathbb{C}\setminus\mathbb{Z}$ 
which is covered by $\mathbb{C}\setminus\mathbb{Z}$ and covers $R_n$ for all $n$.

Thus, in this paper, we would like to investigate quasiconformal deformations of
$\mathbb{C}\setminus\mathbb{Z}$. First, we shall try to find all 
Riemann surfaces which are quasiconformally equivalent to $\mathbb{C}\setminus\mathbb{Z}$.\\

If $R$ is quasiconformally equivalent to $\mathbb{C}\setminus\mathbb{Z}$, 
then $R$ is conformally equivalent to $\mathbb{C}\setminus A$ by a certain closed discrete subset $A\subset \mathbb{C}$ (cf. The removable singularity theorem, see \cite[Theorem\ 17.3.]{vai2}\ and \cite[p.14]{geh2}).
\begin{dfnemp}
A closed discrete subset $A\subset \mathbb{C}$ is called tame if $\mathbb{C}\setminus A$
is quasiconformally equivalent to $\mathbb{C}\setminus \mathbb{Z}$.
\end{dfnemp}
We consider the following problem.
\begin{probemp}
Let $\mathscr{P}$ be the family of all closed discrete infinite subsets $A\subset \mathbb{C}$.
Find all tame $A \in \mathscr{P}$.
\end{probemp}
For this Problem, we obtain the following results as partial solutions.\\

First, for $A \in \mathscr{P},\ z\in \mathbb{C}$, and $r>0$, we define
quantities $d$ and $\tilde{d}$ by
\begin{eqnarray*}
d(z,r;A)\hspace{-2mm}&=&\hspace{-3mm}\sup_{w\in \bar{D}_r(z)} {\rm dist}(w,A),\hspace{3mm} \bar{D}_r(z)=\left\{w\in \mathbb{C} \mid |w-z|\leq r \right\},\\
\tilde{d}(z,r;A)\hspace{-2mm}&=&\hspace{-3mm}\sup_{w\in Q_r(z)} {\rm dist}(w,A),\hspace{3mm} Q_r(z)=\left\{z+x+iy \mid -r\leq x,y \leq r \right\}.
\end{eqnarray*}
\begin{thma} \label{thmA}
Let $A\in \mathscr{P}$. If there exists a closed discrete subset $B\subset \mathbb{R}$
such that $\mathbb{C}\setminus B$ is quasiconformally equivalent
to $\mathbb{C}\setminus A$, then 
\[
\sup_{z\in \mathbb{C},\ r>0}\  \frac{r}{d(z,r;A)}\ ,\ \ \sup_{z\in \mathbb{C},\ r>0}\  \frac{r}{\tilde{d}(z,r;A)}<+\infty.
\]
\end{thma}
\vspace{1ex}
\begin{coroemp}
$\mathbb{Z}+i\mathbb{Z}$ is not tame.
\end{coroemp}

The complex plane $\mathbb{C}$ and the unit disk $\mathbb{D}$ are well known as a typical example of Riemann surfaces which are homeomorphic
and are not quasiconformally equivalent.
This Corollary also gives an example of such Riemann surfaces, but is essentially different from the case of $\mathbb{C}$ and $\mathbb{D}$.

In addition, various corollaries are provided by Theorem\ \ref{thmA}. We shall
introduce them in Section\ \ref{coroA}. To prove Theorem\ \ref{thmA}, the fact
that a quasidisk becomes a uniform domain plays an important role. This fact was 
proved by V. Gol$'$d\v ste\v \i n and S. Vodop$'$janov\ \cite{gol1} first. 
Later, F. W. Gehring and B. Osgood proved the more generalized form in \cite{geh6} and
\cite{osg1}.\\

Next, we obtain the following theorem by comparing the extremal distances 
of certain continua. Let $D(a,r)=\left\{ z\in \mathbb{C} \mid |z-a|<r\right\}$ for
$a \in \mathbb{C}$ and $r>0$.
\begin{thma} \label{thmB}
Let $A\in \mathscr{P}$. Assume there exists an automorphism of infinite order
$h\in {\rm Aut}(\mathbb{C}\setminus A)$ such that 
the quotient space $\left( \mathbb{C}\setminus A\right)/\langle h\rangle$ has
infinitely many punctures. If 
$A$ is tame, then for 
any $\varepsilon>0$ and $d \in \mathbb{N}$, there exists $a\in A$ such that
\[
\# D(a,\varepsilon)\cap A \geq d,
\]
where for a finite set $X$, $\# X$ denotes the number of elements of $X$.
\end{thma} \vspace{0mm}
Recall that $A$ is closed and discrete in $\mathbb{C}$. 
Certainly, there exist some $A\in \mathscr{P}$ 
satisfying the above assumptions 
which cannot be decided whether $A$ is tame or not from Theorem B,
but they have to be a very strange form (see Example\ \ref{exstrange}).
Therefore it is expected that the next proposition holds.
\begin{conjecturer}
Let $A\in \mathscr{P}$. If there exists an automorphism of infinite order
$h\in {\rm Aut}(\mathbb{C}\setminus A)$ such that 
the quotient space $\left( \mathbb{C}\setminus A\right)/\langle h\rangle$ has
infinitely many punctures, then $A$ is not tame.
\end{conjecturer}
We can now rephrase Conjecture I as follows.\\

Let $T_0=\bigcup_{n\in \mathbb{N}} p^n_{\ast} \left( T(R_n)\right)$, namely $T_0$ is a subspace
of $T(\mathbb{C}\setminus\mathbb{Z})$ which simultaneously describes all quasiconformal deformations of all Riemann
surfaces of finite type $(0,n)$ with $n\geq 3$. Further, let $\bar{T}$ be the set of all $[S,f]
\in T(\mathbb{C}\setminus\mathbb{Z})$, the \tei equivalence class of the quasiconformal homeomorphism $f:R\rightarrow S$, such that there exists an automorphism of infinite order
in ${\rm Aut}(S)$. Then Conjecture I implies
\begin{conjecturer}
\hspace{5ex} $\displaystyle \bar{T}=\bigcup_{[f]\in {\rm Mod}(\mathbb{C}\setminus\mathbb{Z})} [f]_{\ast}\left(T_0\right)$.
\end{conjecturer}
Here, ${\rm Mod}(\mathbb{C}\setminus\mathbb{Z})$ is the Teichm\"uller-Modular group of $\mathbb{C}\setminus \mathbb{Z}$.\ 
The subspace $T_0$ is not closed in $T(\mathbb{C}\setminus\mathbb{Z})$. However, $T_0$ is separable by the definition,
further geodesically convex with respect to the \tei metric of $T(\mathbb{C}\setminus\mathbb{Z})$ by McMullen's theorem.
Remark that, Conjecture I and II are not necessary to prove these properties.\

\section{Proof of Theorem\ \ref{thmA}}

\subsection{Quasidisks and uniform domains}
A domain $D \subset \hat{\mathbb{C}}$ is called a quasidisk
if there exists a quasiconformal homeomorphism $f:\hat{\mathbb{C}}
\rightarrow \hat{\mathbb{C}}$ such that $f(\mathbb{D})=D$, where 
$\hat{\mathbb{C}}$ and $\mathbb{D}$ denote the Riemann sphere and the unit disk, respectively. 
Next, for a constant $c\geq 1$, a domain $D \subset \mathbb{C}$ is called a $c$-uniform domain
if arbitrary 
two points $z_1,\ z_2 \in D$ can be joined by a rectifiable curve $\gamma \subset D$ 
which satisfies 
\begin{enumerate}\setlength{\leftskip}{2cm}
\item $\ell(\gamma)\leq c\ |z_1-z_2|$,\vspace{1ex}
\item for all $z\in \gamma$,\ \ $\displaystyle \min _{j=1,2}\ell(\gamma[z,z_j]) \leq c\ {\rm dist}(z,\partial D)$.
\end{enumerate}
Here $\gamma[z,z_j]$ denotes the subcurve of $\gamma$ which joins $z$ and $z_j$.
Further $D$ is called a uniform domain if $D$ is a $c$-uniform domain for some 
constant $c\geq 1$.

There is a lot of characterizations
of quasidisks. Many of those are collected in \cite{geh2} by F. W. Gehring. 
Actually, the uniformity of domains is one of the characterizations of quasidisks:
for a simply connected proper subdomain $D\subset \mathbb{C}$, $D$ is
a quasidisk if and only if $D$ is a uniform domain
(cf. \cite{gol1}, \cite{geh6} and \cite{osg1}).

\subsection{Proof of Theorem \ref{thmA}}
\begin{thma}
Let $A\in \mathscr{P}$. If there exists a closed discrete subset $B\subset \mathbb{R}$
such that $\mathbb{C}\setminus B$ is quasiconformally equivalent
to $\mathbb{C}\setminus A$, then 
\[
\sup_{z\in \mathbb{C},\ r>0}\  \frac{r}{d(z,r;A)}\ ,\ \ \sup_{z\in \mathbb{C},\ r>0}\  \frac{r}{\tilde{d}(z,r;A)}<+\infty.
\]
\end{thma}
\begin{proof} 
$\bar{D}_r(z) \subset Q_r(z) \subset
 \bar{D}_{\sqrt{2}r}(z)$ holds for all $z \in \mathbb{C}$ and $r>0$, so that,
$d(z,r;A) \leq \tilde{d}(z,r;A) \leq d(z,\sqrt{2}r;A). \label{relation}$ Thus
\begin{eqnarray*}
\frac{1}{\sqrt{2}}\ \sup_{z\in \mathbb{C},\ r>0}\  \frac{r}{\tilde{d}(z,r;A)}\leq \sup_{z\in \mathbb{C},\ r>0}\  \frac{r}{d(z,r;A)}\leq \sup_{z\in \mathbb{C},\ r>0}\  \frac{r}{\tilde{d}(z,r;A)}. \label{futoshiki}
\end{eqnarray*} 
Therefore it suffices to prove that Theorem \ref{thmA} holds for $d$.\\

Let $f:\mathbb{C}\setminus B \rightarrow \mathbb{C}\setminus A$ be a quasiconformal 
homeomorphism. 
It is known from the removable singularity theorem that $f$ can be extended to
the quasiconformal homeomorphism from $\mathbb{C}$ to $\mathbb{C}$
(cf. see \cite[Theorem\ 17.3.]{vai2}\ and \cite[p.14]{geh2}).
Moreover $\infty$ is also a removable singularity for extended $f$. Consequently, $f$ can be 
extended to the quasiconformal homeomorphism $f:\hat{\mathbb{C}} \rightarrow \hat{\mathbb{C}}$
(we denote the extended $f$ by the same letter $f$). Then
$D=f(\mathbb{H})$ becomes a quasidisk, where $\mathbb{H}$ denotes
the upper half plane. Since $f(\infty)=\infty$, $f(\mathbb{H})\subset \mathbb{C}$.
Therefore $D$ becomes a $c$-uniform domain for some constant $c\geq 1$. In the following argument,
remark that the restriction $f|_B :B\rightarrow A$ is bijective. \\

Let $z_1 \in D$ and $r>0$. Since $D$ is not bounded, we can choose $z_2 \in D$
such that
\[
|z_1-z_2|=\frac{2r}{c}.
\]
Then from the uniformity of $D$, there exists a rectifiable curve 
$\gamma \subset D$ which joins $z_1$ and $z_2$ with the properties that
\begin{enumerate}\setlength{\leftskip}{2cm}
\item $\ell(\gamma)\leq c\ |z_1-z_2|=2r$,\vspace{1ex}
\item for all $z\in \gamma$,\ \ $\displaystyle \min _{j=1,2}\ell(\gamma[z,z_j]) \leq c\ {\rm dist}(z,\partial D)$.
\end{enumerate}
By using the intermediate value theorem for $\gamma[z,z_1]$, it turns out there exists $z_0\in \gamma$ such
that $\ell(\gamma[z_0,z_1])=\ell(\gamma)/2$. Then
\[
\min_{j=1,2} \ell(\gamma[z_0,z_j])=\ell(\gamma[z_0,z_1])=\ell(\gamma[z_0,z_2]).
\]
Therefore
\begin{eqnarray*}
2r\geq \ell(\gamma)=\ell(\gamma[z_0,z_1])+\ell(\gamma[z_0,z_2])
      =2\ell(\gamma[z_0,z_1]) \geq 2|z_0-z_1|.
\end{eqnarray*}
Namely $z_0 \in \bar{D}_{r}(z_1)$. From the above,
\begin{eqnarray*}
\frac{2r}{c}=|z_1-z_2| &\leq &\ell(\gamma)\\
         &=&2\min_{j=1,2} \ell(\gamma[z_0,z_j])\\
         &\leq & 2c\ {\rm dist}(z_0,\partial D)
         \leq  2c\ {\rm dist}(z_0, A) \leq 2c\ d(z_1,r;A).
\end{eqnarray*}
Thus we obtain $\displaystyle \frac{r}{d(z_1,r;A)}\leq c^2$. Recall that $z_1\in D$ and
$r>0$ are arbitrary. 

Moreover we can apply the above argument to $z_1 \in\partial D \setminus \{\infty\}$ and $r>0$, 
because of V\"ais\"al\"a's theorem \cite[Theorem\ 2.11.]{vai1}.

Finally, in the case of $z_1 \in \mathbb{C}\setminus \bar{D}$ and $r>0$, 
we apply the same argument to the image of the lower half plane under $f$. Then it turns out there exists a constant
$c'\geq 1$ which does not depend on $z_1$ and $r$, and satisfies
\[
\frac{r}{d(z_1,r;A)}\leq c'^2.
\]
Hence we obtain the claim.
\end{proof}

\subsection{Some corollaries of Theorem \ref{thmA}} \label{coroA}
In this section, we shall show some corollaries from Theorem \ref{thmA}.
\begin{coro} \label{coro1}
For arbitrary $s>0$, the discrete set $A_s=\mathbb{Z}+ i\left\{ \pm n^s \mid n=0,1,2,\cdots \right\}$
is not tame. In particular $\mathbb{Z}+i\mathbb{Z}$ is not tame.
\end{coro}
\begin{proof}
If $s\leq 1$, it follows that $\tilde{d}(0,r;A_s)=\sqrt{2}/2$ for all $r>1$ 
(see Figure \ref{zu1}). Therefore
\[
\lim_{r\rightarrow +\infty} \frac{r}{\tilde{d}(0,r;A_s)}=+\infty.
\]
Since $\mathbb{Z}$ is closed and discrete in $\mathbb{R}$, it turns out that $\mathbb{C}\setminus A_s$
is not quasiconformally equivalent to $\mathbb{C}\setminus \mathbb{Z}$ from Theorem \ref{thmA}.
Namely $A_s$ is not tame.

\begin{figure}[htbp]
\begin{tabular}{cc}
 \begin{minipage}{0.5\hsize}
  \begin{center}
   \includegraphics[width=60mm]{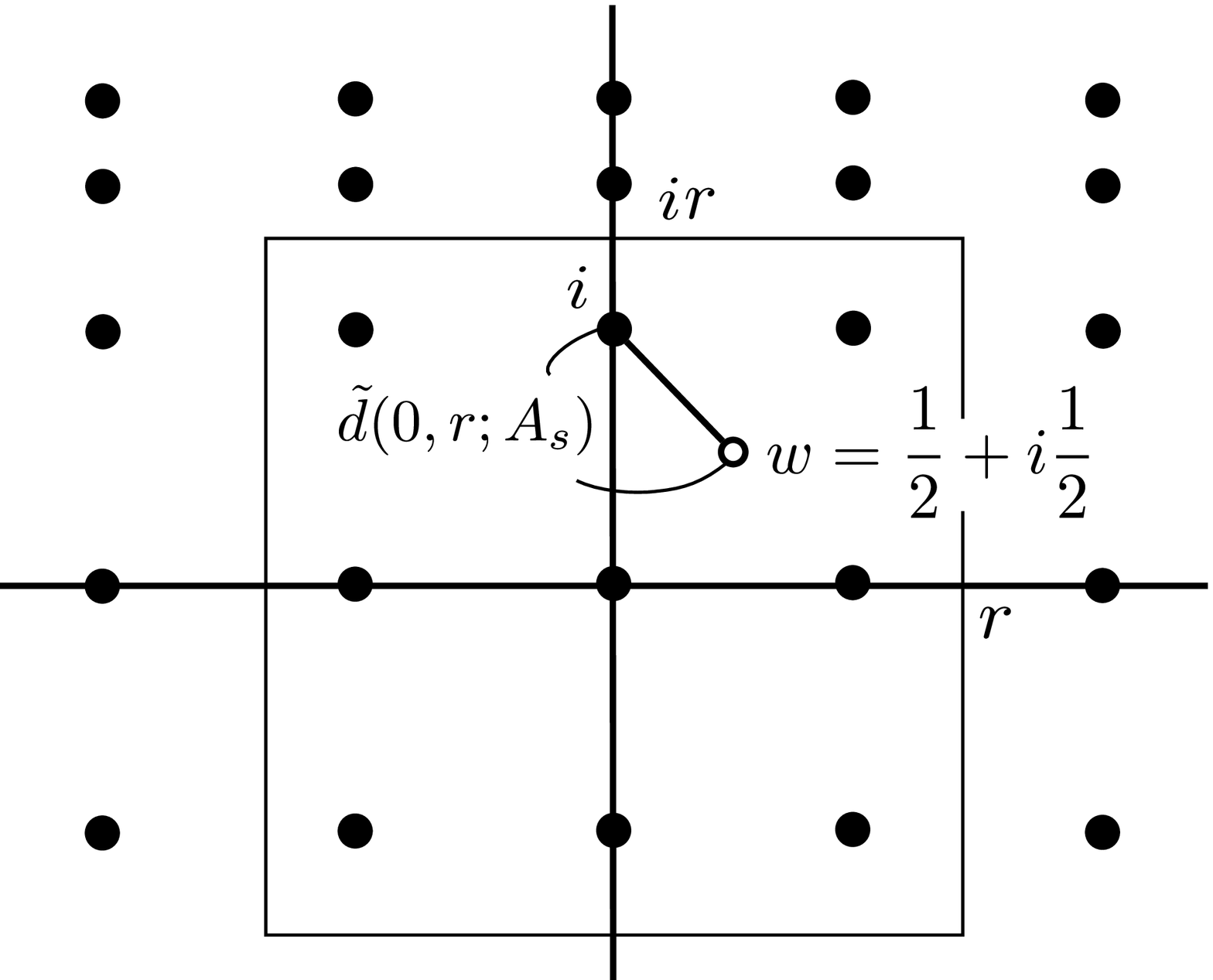}
  \end{center}
  \caption{:　$s\leq 1$}\label{zu1}
 \end{minipage}
 \begin{minipage}{0.5\hsize}
  \begin{center}
   \includegraphics[width=64mm]{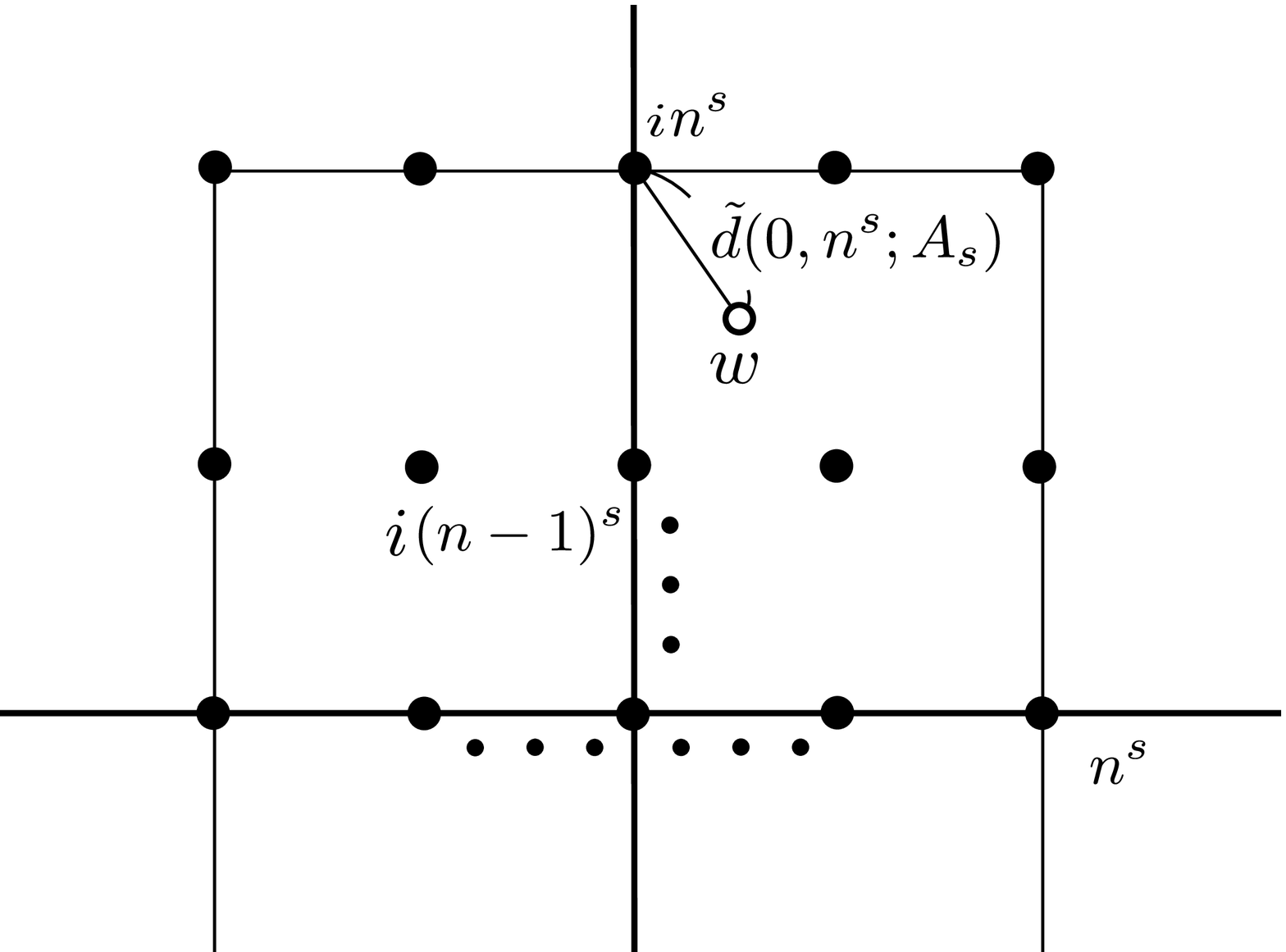}
  \end{center}
  \caption{:　$s>1$}\label{zu2}
 \end{minipage}
\end{tabular}
\end{figure} 

Next, if $s>1$, 
\[
\left\{2\tilde{d}(0,n^s;A_s)\right\}^2= \left(n^s-(n-1)^s\right)^2+1
\]
holds for all $n\in \mathbb{N}$ (see Figure \ref{zu2}). Therefore
\begin{eqnarray*}
\frac{n^s}{\tilde{d}(0,n^s;A_s)}&=&\frac{2n^s}{\sqrt{\left(n^s-(n-1)^s\right)^2+1}}\\
&=&\frac{2}{\sqrt{\displaystyle \left(1-\left(1-\frac{1}{n}\right)^s\right)^2+\left(\frac{1}{n}\right) ^{2s}}} \rightarrow +\infty\ \ (n\rightarrow +\infty).
\end{eqnarray*}
Similarly it is revealed that $A_s$ is not tame.
\end{proof}
\newpage
\begin{coro}
For arbitrary $s>0$, the discrete set $A'_s=\mathbb{Z}+i\left\{ n^s \mid n=0,1,2,\cdots \right\}$
is not tame. In particular $\mathbb{Z}+i\mathbb{N}$ is not tame.
\end{coro}
\begin{proof}
Instead of $\tilde{d}(0,n^s;A_s)$ used in the proof of Corollary \ref{coro1},
we compute $\displaystyle \tilde{d}(n^si/2,n^s/2;A'_s)$. It follows that
\begin{eqnarray*}
\frac{\displaystyle \frac{n^s}{2}}{\tilde{d}\left(\displaystyle \frac{n^s}{2}i,\frac{n^s}{2};A'_s\right)}&=&\frac{n^s}{\sqrt{\left(n^s-(n-1)^s\right)^2+1}}\\
&=&\frac{1}{\sqrt{\displaystyle \left(1-\left(1-\frac{1}{n}\right)^s\right)^2+\left(\frac{1}{n}\right) ^{2s}}} \rightarrow +\infty\ \ (n\rightarrow +\infty).
\end{eqnarray*}
\end{proof}

\subsection{Example} \label{exa}
Let $A=\mathbb{Z}+i\left\{2^n \mid n=0,1,2,\cdots \right\}$. It seems that $A$ is similar to
$A_s$ and $A'_s$ of the above corollaries, however, we cannot decide whether
$A$ is tame or not from Theorem \ref{thmA}.

\section{Proof of Theorem\ \ref{thmB}}
\subsection{Extremal distances and Vuorinen's Theorem}
Let $D\subset \hat{\mathbb{C}}$ be a domain. For given continua $E, F\subset D$,
\[
\edd{E}{F}=\mode(\zokud{E}{F})
\]
is called the extremal distance between $E$ and $F$ in $D$, where $\mode$
denotes the $n$-modulus of a curve family and $\zokud{E}{F}$ denotes the family
of all rectifiable curves which join $E$ and $F$ in $D$. The $n$-modulus coincides with
the reciprocal of the extremal length introduced by L.\ V.\ Ahlfors and A.\ Beurling \cite{ahl3}.
It is well known that a sense preserving homeomorphism $f$ becomes $K$-quasiconformal for a constant $K\geq1$ if
and only if $f$ satisfies the following inequality for any curve family $\mathscr{F}$ in the domain of $f$.
\[
\frac{1}{K}\mode(\mathscr{F})\leq \mode(f(\mathscr{F})) \leq K\mode(\mathscr{F}).
\]

 The next useful lower bound for extremal distances was presented by M.\ Vuorinen in
 \cite[Lemma\ 4.7]{vuo3}. For each pair of disjoint continua $E, F\subset \hat{\mathbb{C}}$,
 it holds that
\[
\edc{E}{F} \geq \frac{2}{\pi} \log \left( 1+ \frac{\min\{{\rm diam}(E),{\rm diam}(F)\}}{{\rm dist}(E,F)} \right).
\]

\subsection{Some lemmas}
To prove Theorem \ref{thmB}, we shall prove some lemmas.
\begin{lemma} \label{lemma1}
Let $A\in \mathscr{P}$ and $h \in {\rm Aut}(\mathbb{C}\setminus A)$.
If ${\rm ord}(h)=\infty$, then $h$ can be written as $h(z)=z+b$ with
a certain $b\in \mathbb{C}\setminus \{0\}$.
\end{lemma}
\begin{proof}
From Riemann's removable singularity theorem, we can regard $h$ as an element of 
${\rm Aut}(\mathbb{C})$. Therefore $h$ can be written as $h(z)=az+b$ with certain
$a\in \mathbb{C}\setminus\{0\}$ and $b\in \mathbb{C}$. It is easily seen $a=1$ and $b\neq 0$
from the conditions ${\rm ord}(h)=\infty$ and $A\in \mathscr{P}$.
\end{proof}
\begin{lemma} \label{lemma2}
Let $A\in \mathscr{P}$. Assume that $z+1 \in {\rm Aut}(\mathbb{C}\setminus A)$ and
$(\mathbb{C}\setminus A)/\langle z+1\rangle$ has infinitely many punctures.
If there exists a quasiconformal homeomorphism $f:\mathbb{C}\setminus \mathbb{Z}\rightarrow\mathbb{C} \setminus A$,
then $f$ satisfies the following condition.
\[
\sup_{a\in A}\ \left| f^{-1}\left( a\right)-f^{-1}\left( a-1 \right)\right|=+\infty.
\]
Here, we use the same letter $f$ for the quasiconformal extension of $f$ which maps $\mathbb{C}$ onto $\mathbb{C}$ (see the proof of Theorem \ref{thmA}).
\end{lemma}
Remark that since $z+1$ belongs to ${\rm Aut}(\mathbb{C}\setminus A)$, \  $a-1\in A$ for all $a\in A$. Further, since $f(\mathbb{Z})=A$, \ $f^{-1}(a)\in \mathbb{Z}$ for all $a \in A$.
\begin{proof}
To obtain a contradiction, assume that there exists a constant $M\in \mathbb{N}$ such that $|f^{-1}(a)-f^{-1}(a-1)|\leq M$ for all $a\in A$.

Let $S=\{x+iy \mid x\in [0,1),\ y\in \mathbb{R}\}$. The assumption that
$(\mathbb{C}\setminus A)/\langle z+1\rangle$ has infinitely many punctures means 
that $S\cap A$ consists of countably infinitely many points. Recall that $A$ is closed and discrete in $\mathbb{C}$,
so that, $\{ {\rm Im}(z) \mid z\in S\cap A\}$ also consists of countably infinitely many points.
Numbering them suitably, we let
\[
\left\{ \left.{\rm Im}(z)\right| z\in S\cap A\right\}=\{a_n\}_{n=1}^{\infty},
\]
and define curves $C_n,\ C'_n$ by
\[
C_n(t)=t+ia_n \hspace{4mm}(t\in \mathbb{R}), \hspace{6mm} C'_n=f^{-1}(C_n).
\]
By definition, each curve $C'_n$ passes integers. Conversely, for each $m\in \mathbb{Z}$, there uniquely exists a curve
in $\{ C'_n\}_{n=1}^{\infty}$ which passes $m$.
We denote such a curve by $C'_{n_m}$. 
\begin{figure}[h]
\begin{center}
\includegraphics[width=12.5cm]{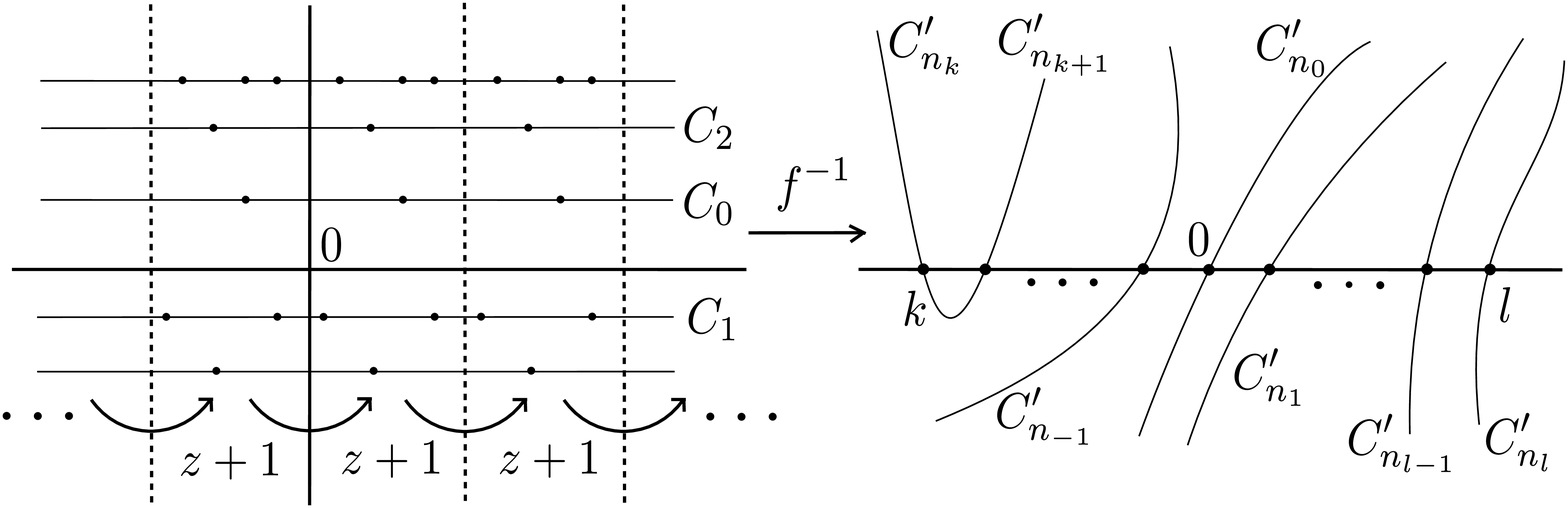} 
\end{center}
\end{figure}

From the properties of the curve family $\{C'_n\}_{n=1}^{\infty}$, there exist $k,l\in\mathbb{Z}\ (k<l)$ such that $\{ C'_{n_{k}},\ C'_{n_{k+1}},\cdots,\ C'_{n_{l-1}},\ C'_{n_l}\}$
consists of $M+1$ curves.
Since $|f^{-1}(a)-f^{-1}(a-1)|\leq M$ for all $a\in A$,
each curve $C'_{n_i}\ (i=k,k+1,\cdots,l)$ passes at least two points of
$\{k-M,k-M+1,\cdots,k-1\}\cup \{l+1,l+2,\cdots,l+M\}$. There must exist 
at least $2(M+1)$ points to pass, however, $\{k-M,k-M+1,\cdots,k-1\}\cup \{l+1,l+2,\cdots,l+M\}$ consists of only $2M$ points. This is a contradiction.
\end{proof}
\newpage
\begin{lemma} \label{lemma3}
Let $A \in \mathscr{P}$ be tame and $f:\mathbb{C}\setminus \mathbb{Z}\rightarrow\mathbb{C}
\setminus A$ be a $K$-quasiconformal homeomorphism. For arbitrary $n,m \in \mathbb{Z}
\ (m>n)$
and $d\in \mathbb{N}$, we set $N_d=[n-d,n]$ and $M_d=[m,m+d]$. Then
\[
\frac{\min\left\{ {\rm diam} f(N_d),{\rm diam}f(M_d) \right\}}{\left|f(m)-f(n)\right|} \leq  \exp \left( \frac{\pi^2 K}{\displaystyle \log \left( 1+2\frac{m-n}{d} \right)} \right)-1.
\]
Here, we use the same letter $f$ for the quasiconformal extension of $f$ which maps $\mathbb{C}$ onto $\mathbb{C}$ (see the proof of Theorem \ref{thmA}).
\end{lemma}
\begin{proof}
Because ${\rm dist}(f(N_d),f(M_d)) \leq |f(m)-f(n)|$, it follows from Vuorinen's Theorem that
\[
\frac{2}{\pi}\log \left(1+\frac{\min\left\{ {\rm diam}f(N_d),{\rm diam}f(M_d)\right\}}{|f(m)-f(n)|}\right) \leq \edc{f(N_d)}{f(M_d)}.
\]
The quasiconformality of $f$ implies that
\[
\edc{f(N_d)}{f(M_d)}\leq K\edc{N_d}{M_d}.
\]
Further, since $N_d,M_d$ are separated by the ring domain $\{ d/2 <|z-(n-d/2)|<d/2+m-n\}$,
\[
\edc{N_d}{M_d}\leq \frac{2\pi}{\displaystyle \log \frac{d/2+m-n}{d/2}}=\frac{2\pi}{\displaystyle \log \left(1+2\frac{m-n}{d}\right)}.
\]
Consequently, we obtain
\[
\frac{2}{\pi}\log \left(1+\frac{\min\left\{ {\rm diam}f(N_d),{\rm diam}f(M_d)\right\}}{|f(m)-f(n)|}\right)  \leq \frac{2\pi K}{\displaystyle \log \left(1+2\frac{m-n}{d}\right)}.
\]
Simplifying this inequality, we obtain the required inequality. 
\end{proof}

\subsection{Proof of Theorem \ref{thmB}}
Theorem \ref{thmB} is proved as a corollary of the next theorem.
\begin{thm} \label{thm}
Let $A\in \mathscr{P}$. Assume there exists an automorphism of infinite order
$h\in {\rm Aut}(\mathbb{C}\setminus A)$ such that
the quotient space $\left( \mathbb{C}\setminus A\right)/\langle h\rangle$ has
infinitely many punctures. If there exists a quasiconformal homeomorphism 
$f:\mathbb{C}\setminus\mathbb{Z} \rightarrow \mathbb{C}\setminus A$, then
for arbitrary $d\in \mathbb{N}$ and $\varepsilon>0$, there exists $n\in \mathbb{Z}$
such that
\[
{\rm diam}f([n-d,n]) \leq \varepsilon.
\]
Here, we use the same letter $f$ for the quasiconformal extension of $f$ which maps $\mathbb{C}$ onto $\mathbb{C}$ (see the proof of Theorem \ref{thmA}).
\end{thm}
\begin{proof}
By Lemma \ref{lemma1}, $h$ can be written as $h(z)=z+b$ with a certain 
$b\in\mathbb{C}\setminus \{0\}$. Then there exists an affine transformation $g\in{\rm Aut}(\mathbb{C})$ 
satisfying $g\circ h\circ g^{-1}(z)=z+1 \in {\rm Aut}(\mathbb{C}\setminus g(A))$.
Therefore we may assume $h(z)=z+1$.

Let $d\in \mathbb{N},\ \varepsilon>0$ and $f:\mathbb{C}\setminus \mathbb{Z}\rightarrow \mathbb{C}\setminus A$ be a $K$-quasiconformal homeomorphism.
Choose $\varepsilon'>0$ with 
\[
 \exp \left( \frac{\pi^2 K}{\displaystyle \log \left( 1+2/d\varepsilon' \right)} \right)-1<\varepsilon.
 \]
By Lemma \ref{lemma2}, there exists $a\in A$ such that
\[
\left| f^{-1}(a)-f^{-1}(a-1) \right| > \frac{1}{\varepsilon'}.
\]
Let $n=\min\{f^{-1}(a),f^{-1}(a-1)\},\ m=\max\{f^{-1}(a),f^{-1}(a-1)\}$. 
Since $|f(m)-f(n)|=1$, it follows from Lemma \ref{lemma3} that
\begin{eqnarray*}
\min\left\{ {\rm diam} f(N_d),{\rm diam}f(M_d) \right\} &\leq & \exp \left( \frac{\pi^2 K}{\displaystyle \log \left( 1+2\frac{m-n}{d} \right)} \right)-1\\
      &\leq & \exp \left( \frac{\pi^2 K}{\displaystyle \log \left( 1+2/d\varepsilon' \right)} \right)-1<\varepsilon,
\end{eqnarray*}
where $N_d=[n-d,n]$ and $M_d=[m,m+d]$.

Therefore if $\min\left\{ {\rm diam} f(N_d),{\rm diam}f(M_d) \right\}={\rm diam} f(N_d)$, then $n$ is the desired integer. In the other case, $m+d$ is the desired integer.
\end{proof}
\setcounter{theo}{1}
\begin{thma} \label{thmB}
Let $A\in \mathscr{P}$. Assume there exists an automorphism of infinite order
$h\in {\rm Aut}(\mathbb{C}\setminus A)$ such that
the quotient space $\left( \mathbb{C}\setminus A\right)/\langle h\rangle$ has
infinitely many punctures. If 
$A$ is tame, then for 
any $\varepsilon>0$ and $d \in \mathbb{N}$, there exists $a\in A$ such that
\[
\# D(a,\varepsilon)\cap A \geq d,
\]
where for a finite set $X$, $\# X$ denotes the number of elements of $X$.
\end{thma}
\begin{proof}
In the proof of Theorem \ref{thm},
 remark that the restriction $\left.f\right|_{\mathbb{Z}}:\mathbb{Z}\rightarrow A$
is bijective. We immediately obtain the claim by taking $a=f(n)$ for $n$ obtained in Theorem \ref{thm}.
\end{proof}

\subsection{Example} It directly follows from Theorem \ref{thmB} that the closed discrete subset $A$ of Example \ref{exa}
is not tame.

\subsection{Example} \label{exstrange}
Let
\[
\displaystyle A=\mathbb{Z}+i\bigcup_{n=1}^{\infty}\left\{ \left.2^n+\left(\frac{1}{2}\right)^{n+1}e^{\frac{2\pi ki}{n}}\right|
k=0,1,\cdots,n-1 \right\}.
\]
Then $A$ satisfies the conditions that $z+1\in {\rm Aut}(\mathbb{C}\setminus A)$ and
$(\mathbb{C}\setminus A)/\langle z+1\rangle$ has infinitely many puctures, however,
it cannot be decided whether $A$ is tame or not from Theorem \ref{thmA} and
Theorem \ref{thmB}.
    



\begin{bibdiv}
\begin{biblist}

\bib{ahl3}{article}{
      author={Ahlfors, L.~V.},
       title={Quasiconformal reflections},
        date={1963},
        ISSN={0001-5962},
     journal={Acta Math.},
      volume={109},
       pages={291\ndash 301},
      review={\MR{0154978 (27 \#4921)}},
}

\bib{geh2}{incollection}{
      author={Gehring, F.~W.},
       title={Characterizations of quasidisks},
        date={1999},
   booktitle={Quasiconformal geometry and dynamics ({L}ublin, 1996)},
      series={Banach Center Publ.},
      volume={48},
   publisher={Polish Acad. Sci., Warsaw},
       pages={11\ndash 41},
      review={\MR{1709972 (2000g:30014)}},
}

\bib{geh6}{article}{
      author={Gehring, F.~W.},
      author={Osgood, B.~G.},
       title={Uniform domains and the quasihyperbolic metric},
        date={1979},
        ISSN={0021-7670},
     journal={J. Analyse Math.},
      volume={36},
       pages={50\ndash 74 (1980)},
         url={http://dx.doi.org.ejgw.nul.nagoya-u.ac.jp/10.1007/BF02798768},
      review={\MR{581801 (81k:30023)}},
}

\bib{gol1}{article}{
      author={Gol{$'$}d{\v{s}}te{\u\i}n, V.~M.},
      author={Vodop{$'$}janov, S.~K.},
       title={Prolongement des fonctions de classe {$L^{1}_{p}$} et
  applications quasi conformes},
        date={1980},
        ISSN={0151-0509},
     journal={C. R. Acad. Sci. Paris S\'er. A-B},
      volume={290},
      number={10},
       pages={A453\ndash A456},
      review={\MR{571380 (81k:30024)}},
}

\bib{mcm1}{article}{
      author={McMullen, C.},
       title={Amenability, {P}oincar\'e series and quasiconformal maps},
        date={1989},
        ISSN={0020-9910},
     journal={Invent. Math.},
      volume={97},
      number={1},
       pages={95\ndash 127},
         url={http://dx.doi.org.ejgw.nul.nagoya-u.ac.jp/10.1007/BF01850656},
      review={\MR{999314 (90e:30048)}},
}

\bib{osg1}{article}{
      author={Osgood, B.~G.},
       title={Univalence criteria in multiply-connected domains},
        date={1980},
        ISSN={0002-9947},
     journal={Trans. Amer. Math. Soc.},
      volume={260},
      number={2},
       pages={459\ndash 473},
         url={http://dx.doi.org.ejgw.nul.nagoya-u.ac.jp/10.2307/1998015},
      review={\MR{574792 (81h:30021)}},
}

\bib{vai2}{book}{
      author={V{\"a}is{\"a}l{\"a}, J.},
       title={Lectures on {$n$}-dimensional quasiconformal mappings},
      series={Lecture Notes in Mathematics, Vol. 229},
   publisher={Springer-Verlag, Berlin-New York},
        date={1971},
      review={\MR{0454009 (56 \#12260)}},
}

\bib{vai1}{article}{
      author={V{\"a}is{\"a}l{\"a}, J.},
       title={Uniform domains},
        date={1988},
        ISSN={0040-8735},
     journal={Tohoku Math. J. (2)},
      volume={40},
      number={1},
       pages={101\ndash 118},
  url={http://dx.doi.org.ejgw.nul.nagoya-u.ac.jp/10.2748/tmj/1178228081},
      review={\MR{927080 (89d:30027)}},
}

\bib{vuo3}{article}{
      author={Vuorinen, M.},
       title={On {T}eichm\"uller's modulus problem in {${\bf R}^n$}},
        date={1988},
        ISSN={0025-5521},
     journal={Math. Scand.},
      volume={63},
      number={2},
       pages={315\ndash 333},
      review={\MR{1018820 (90k:30038)}},
}

\end{biblist}
\end{bibdiv}



\end{document}